\documentclass[11pt]{amsart}
\usepackage{amssymb,amscd,  latexsym, graphicx, mathrsfs, enumerate}

\numberwithin{equation}{section}
\newtheorem{thm}{Theorem}[section]
\newtheorem{prop}[thm]{Proposition}
\newtheorem{ques}[thm]{Question}

\newtheorem{remark}[thm]{Remark}
\newtheorem{lem}[thm]{Lemma}

\newtheorem{cor}[thm]{Corollary}

\newtheorem{conjecture}[thm]{Conjecture}

\newtheorem{assump}{Assumption}[section]
\theoremstyle{definition}

\newcommand{\bQ}{\overline{\mathbb{Q}}}

\newcommand{\C}{\mathbb{C}}
\newcommand{\R}{\mathbb{R}}
\newcommand{\Q}{\mathbb{Q}}

\newcommand{\F}{\mathbb{F}}
\newcommand{\A}{\mathbb{A}}

\newcommand{\diag}{{\rm diag}}

\newcommand{\brho}{\overline{\rho}}

\newcommand{\ds}{\displaystyle}

\newcommand{\bF}{\overline{\F}}

\newcommand{\lra}{\longrightarrow}

\renewcommand{\Bbb}{\mathbb}

\title[A conditional construction of Artin representations of $GL_n$]
{A uniform structure on subgroups of $GL_n(\F_q)$ and \\ its application to
a conditional construction of \\ Artin representations of $GL_n$} 
\author{Henry H. Kim and Takuya Yamauchi}
\thanks{The first author is partially supported by NSERC. The second author
is partially supported by JSPS Grant-in-Aid for Scientific Research (C) No.15K04787}
\subjclass[2010]{}
\address{Henry H. Kim \\
Department of mathematics \\
 University of Toronto \\
Toronto, Ontario M5S 2E4, CANADA \\
and Korea Institute for Advanced Study, Seoul, KOREA}
\email{henrykim@math.toronto.edu}

\address{Takuya Yamauchi \\
Department of mathematics, Faculty of Education\\
Kagoshima University\\
Korimoto 1-20-6 Kagoshima 890-0065, JAPAN}
\email{yamauchi@edu.kagoshima-u.ac.jp}

\begin{document}

\begin{abstract}
Continuing our investigation in \cite{kim-yam}, where we associated an Artin representation to a vector-valued real analytic Siegel cusp form of weight $(2,1)$ under reasonable assumptions, 
we associate an Artin representation of $GL_n$ to a cuspidal representation of $GL_n(\A_\Q)$ with similar assumptions.
A main innovation in this paper is to obtain a uniform structure of subgroups in $GL_n(\F_q)$, which enables us to avoid complicated  case by case analysis in \cite{kim-yam}. We also supplement \cite{kim-yam} by showing that we can associate non-holomorphic Siegel modular forms of weight $(2,1)$ to Maass forms for $GL_2(\A_\Q)$ and to 
cuspidal representations of $GL_2(\A_K)$ over imaginary quadratic fields $K$.
\end{abstract}

\maketitle
\tableofcontents

\section{Introduction}
This paper is a continuation of \cite{kim-yam}, where we associated an irreducible complex Galois representation, called Artin representation, $\rho: G_{\Bbb Q}:={\rm Gal}(\bQ/\Q)\longrightarrow GSp_4(\Bbb C)$, to a vector-valued real analytic Siegel cusp form of weight $(2,1)$, under some reasonable assumptions, by generalizing the result of Deligne and Serre \cite{d&s} who associated an odd irreducible Artin representation 
$\rho_f: G_\Q\lra GL_2(\C)$ to any elliptic cusp form $f$ of weight one. Contrary to the case of holomorphic modular forms of weight one, real analytic Siegel cusp forms of weight $(2,1)$ do not have algebro-geometric structures, and thus several assumptions are needed to carry out Deligne-Serre construction. In this paper, we associate an Artin representation to a (unitary) cuspidal representation of $GL_n(\Bbb A_{\Bbb Q})$ with similar assumptions. 

More precisely, let $\pi$ be a (unitary) cuspidal representation of $GL_n(\A_{\Bbb Q})$ with the central character 
$\omega$.
Let $N$ be the conductor of $\pi$ so that $\pi_p$ is unramified for $p\nmid N$. Let $\{\alpha_{1}(p),...,\alpha_{n}(p)\}$ be the Satake parameters for $p\nmid N$. 
Let
$$H_p(T):=(1-\alpha_{1}(p)T)\cdots (1-\alpha_{n}(p)T)=1-a_1(p)T+\cdots+ (-1)^n a_n(p)T^n,
$$ 
be the Hecke polynomial for $p\nmid N$. Then
\begin{eqnarray*}
&& a_1(p)=\alpha_{1}(p)+\cdots+\alpha_{n}(p),\quad a_2(p)=\sum_{1\leq i<j\leq n} \alpha_{i}(p)\alpha_{j}(p),\dots,\\
&& a_m(p)=\sum_{i_1<\cdots<i_m} \alpha_{i_1}(p)\cdots\alpha_{i_m}(p),\cdots, a_n(p)=\alpha_1(p)\cdots \alpha_n(p)=\omega(p).
\end{eqnarray*}

Let $\Bbb Q_\pi=\Bbb Q(a_m(p), m=1,...,n, p\nmid N)$ be the Hecke field of $\pi$. 
Let $K$ be the Galois closure of $\Bbb Q_\pi$, and $\mathcal O_K$ be the ring of integers of $K$. 

We assume the following:
\begin{enumerate}
\item (Finiteness of Hecke fields) $\Bbb Q_\pi$ is a finite extension of $\Bbb Q$, i.e., $K$ is a finite extension of $\Bbb Q$;
\item (Integrality of Hecke polynomials) There exists an integer $M>0$ such that $M a_m(p)\in \mathcal O_K$ for $m=1,...,n$ and $p\nmid N$;
\item (Existence of Galois conjugates) For each $\sigma\in {\rm Gal}(K/\Bbb Q)$, ${}^\sigma \pi$ is a cuspidal representation of $GL_n(\A_\Q)$ with conductor $N_{\sigma}$ such that for $p\nmid NN_\sigma$, the Satake parameters are $\{\sigma(\alpha_{1}(p)),...,\sigma(\alpha_{n}(p))\}$;
\item Existence of mod $\ell$ Galois representation attached to $\pi$;
\item (Rankin-Selberg $L$-functions) For each $m=1,...,n$, and each $\sigma\in {\rm Gal}(K/\Bbb Q)$, 
$$\sum_{p\nmid NN_\sigma} \frac {|\sigma(a_m(p)|^2}{p^s}
\leq C_{n,m}^2 \log\frac 1{s-1}+O(1),\quad \text{as $s\to 1^+$},
$$
where $C_{n,m}=\begin{pmatrix} n\\m\end{pmatrix}:=\ds\frac{n!}{m!(n-m)!}$.
\end{enumerate}

These assumptions (1)-(4) are very natural, and are valid for cuspidal representations attached to elliptic cusp forms of weight one.
Note that we assume that the Satake parameters $\alpha_i(p)$'s themselves are integral. It enables us to show that the Satake parameters take only finitely many values by Assumption (5) (Proposition \ref{density}). Recall that in the case of
holomorphic Siegel cusp forms of weight $(k_1,k_2)$, $k_1\geq k_2\geq 2$, Hecke eigenvalues are algebraic integers, but Satake parameters are twisted by $p^{-\frac {k_1+k_2-3}2}$.
Assumption (5) is also natural. 
The $L$-function $\sum_p \frac {|\sigma(a_m(p)|^2}{p^s}$ is closely related to the Rankin-Selberg $L$-function of the exterior $m$-th power $\wedge^m({}^\sigma \pi)$. In the appendix, we describe the relationship, and show that the Langlands functoriality of the exterior $m$-th power $\wedge^m(\pi)$ as an automorphic representation of $GL_{C_{n,m}}(\Bbb A_\Q)$ implies Assumption (5). 

Then we prove the following main theorem:
\begin{thm} [Main Theorem]\label{artin-rep}
Let $\pi$ satisfy the above five assumptions. 
Then there exists the Artin representation 
$\rho_\pi: G_\Q\lra GL_n(\C)$ which is unramified for $p\nmid N$ such that $$\det(I_n-\rho_\pi({\rm Frob}_p)T)=H_p(T)$$ 
for all $p\nmid N$. Furthermore, $\rho_{\pi}(c)\stackrel{\tiny{GL_n(\C)}}{\sim }{\rm diag}(\epsilon_1,\ldots,\epsilon_n),\ 
\epsilon_i\in\{\pm1\}$ for the complex conjugate $c$, and  
$\pi_\infty\simeq \pi(\epsilon'_1,\ldots,\epsilon'_n)$ where $\epsilon'_i=\begin{cases} 1, &\text{if $\epsilon_i=1$}\\ \text{sgn}, &\text{if $\epsilon_i=-1$}\end{cases}.$
\end{thm}

As a corollary, we obtain that the above assumptions on $\pi$ imply the Ramanujan conjecture for $\pi$, namely, 
$\pi_p$ is tempered for all $p$.

The assumptions force $\pi_\infty$ to be a principal series representation of the form 
$\pi(\epsilon'_1,\ldots,\epsilon'_n)$. Naively we hope that this kind of automorphic 
representation $\pi$ should satisfy the above assumptions (1)-(5). However,  
unlike holomorphic modular forms in $GL_2(\A_\Bbb Q)$ case where one can use algebraic geometry as Deligne and Serre did,
it seems difficult in general $GL_n$ case to verify 
whether a given $\pi$ satisfies these strong assumptions. Note that in holomorphic modular forms in $GL_2$ case, $\pi_{\infty}$ is a limit of discrete series. But in the case of $GL_n$, $n\geq 3$, $\pi_{\infty}$ is not a limit of discrete series (cf. \cite{Kn}).

One key ingredient in Deligne-Serre construction is to obtain bounds of orders of certain subgroups of $GL_2(\Bbb F_{\ell^n})$ for any odd prime $\ell$. In \cite{d&s}, it was done by classification of semisimple subgroups of $GL_2(\Bbb F_{\ell^n})$ and case by case analysis. In \cite{kim-yam}, we carried out the same thing by using the classification of semisimple subgroups of $GSp_4(\Bbb F_{\ell^n})$ and case by case analysis.
A main innovation in this paper is to prove some structure theorem (see Section \ref{finite-groups}) 
for semisimple subgroups of $GL_n(\Bbb F_q)$ for any finite field $\F_q$ by using result of Larsen-Pink \cite{lp} 
combined with the appendix in \cite{ghtt}. 
This enables us to generalize Proposition 7.2 of \cite{d&s} to general linear groups without explicit forms of 
subgroups in question at hand. So we can avoid complicated case by case analysis as we have done in \cite{kim-yam}.
We remark that this is not at all obvious from the existing results in the literature (cf. \cite{lp} and \cite{N}).

The organization of this paper is as follows. In Section 2, we investigate the infinity type of an  
automorphic representation of $GL_n(\A_\Q)$ which gives rise to 
an Artin representation. 
In Section 3, 
we apply the Rankin-Selberg method to prove that the number of Satake parameters outside a certain infinite set of primes is finite. 
By using the results in Section 4 with the finiteness of Satake parameters, we bound the size of the image of mod $\ell$ Galois representations, provided that they exist. 
The formulation of the conjecture for the existence of such mod $\ell$ Galois representations is given in Section 5. 
The proof of the main theorem is given in Section 6 following Deligne-Serre. 

In Sections 7 and 8, we recall our previous work \cite{kim-yam} for the case of $GSp_4/\Q$. 
We discuss a relation between non-holomorphic Siegel modular forms and holomorphic Siegel modular forms.    
After the completion of the previous work, we realized that we do not need to use the unproven hypothesis on 
the existence of the weak transfer from $GSp_4$ to $GL_4$ as in \cite{arthur}. We explain how to get around this. 
Finally in Sections 9 and 10, we associate non-holomorphic Siegel modular forms of weight $(2,1)$ to
automorphic representations of $GL_2$ over imaginary quadratic fields, and Maass forms for $GL_2(\A_\Q)$.

\smallskip

\textbf{Acknowledgments.} We would like to thank R. Guralnick and F. Herzig for helpful discussions. 
We thank the referee for the careful reading of the paper and many comments.

\section{Infinity type of Artin representation}

We show that the cuspidal representation $\pi$ of $GL_n(\A_\Q)$ which we are considering has a very special infinity type $\pi_{\infty}$.

\begin{prop}\label{infinity} 
Let $\rho: G_\Q\longrightarrow GL_n(\C)$ be an irreducible continuous
Galois representation, and let $\pi$ be a cuspidal automorphic representation of $GL_n(\Bbb A_\Q)$
such that 
$$\det(I_n-\rho_F({\rm Frob}_p)T)=H_p(T)$$ 
for all $p\nmid N$, where $N$ is the conductor of $\pi$. Then $\pi_{\infty}$ is the full induced representation $\pi(\epsilon'_1,...,\epsilon'_n)$, where $\rho_{\pi}(c)\stackrel{\tiny{GL_n(\C)}}{\sim }{\rm diag}(\epsilon_1,\ldots,\epsilon_n),\ 
\epsilon_i\in\{\pm1\}$ for the complex conjugate $c$, and
$\epsilon'_i=\begin{cases} 1, &\text{if $\epsilon_i=1$}\\ \text{sgn}, &\text{if $\epsilon_i=-1$}\end{cases}.$
\end{prop}
\begin{proof} From the assumption, it is clear that
$L(s,\pi_p) = L(s,\rho_p)$ for almost all $p$. Then by Proposition A.1 of \cite{martin},
$L(s, \pi_p) = L(s, \rho_p)$ for all $p$, and $L(s,\pi_\infty)=L(s,\rho_\infty)$. 
Since $\rho_{\pi}(c)\stackrel{\tiny{GL_n(\C)}}{\sim }{\rm diag}(\epsilon_1,\ldots,\epsilon_n),\ 
\epsilon_i\in\{\pm1\}$ for the complex conjugate $c$, the Langlands parameter of $\pi_{\infty}$ is 
$$\phi: W_{\Bbb R}=\Bbb C^\times\cup j \Bbb C^{\times}\longrightarrow GL_n(\Bbb C),\quad \phi(z)=Id,\quad 
\phi(j)=\diag(\epsilon_1,...,\epsilon_n),
$$
where $\epsilon_i\in \{\pm1\}$. Our result follows from this observation. 
\end{proof}

\section{Application of the Rankin-Selberg method}

We prove that Satake parameters take only finitely many values under Assumptions (1)-(3) and (5) in the introduction.

\begin{prop}\label{density} Suppose $\pi=\otimes'_p \pi_p$ is a cuspidal representation of $GL_n(\A_\Q)$ with conductor $N$ which satisfies Assumptions (1)-(3) and (5) in the introduction. Then for any positive integer $\eta$, there exists a set $X_{\eta}$ of rational primes such that den.sup$X_{\eta}\leq \eta$, and the set $\{(a_1(p),...,a_n(p))\, |\, p\notin X_{\eta}\}$ is a finite set, or equivalently, $\{ \text{Satake parameters at $p$}\, |\, p\notin X_{\eta}\}$ is finite.
\end{prop}
Here den.sup$X_{\eta}$ is defined by
$$\limsup_{s\to 1^+} \frac {\sum_{p\in X_{\eta}} p^{-s}}{\log \frac 1{s-1}}.
$$
We also define the Dirichlet density den$(X_{\eta})$ by 
$$\lim_{s\to 1^+} \frac {\sum_{p\in X_{\eta}} p^{-s}}{\log \frac 1{s-1}}.
$$
\begin{proof} 
By Assumption (3), for each $\sigma\in {\rm Gal}(K/\Bbb Q)$, ${}^\sigma \pi$ is a cuspidal representation of $GL_n(\A_\Q)$ with the conductor $N_{\sigma}$ such that for $p\nmid NN_{\sigma}$, the Satake parameters are $\{\sigma(\alpha_{1}(p)),...,\sigma(\alpha_{n}(p))\}$. 
Hence for each $m=1,...,n$ and each $\sigma$, by Assumption (5),
$$\sum_{p\nmid NN_\sigma} \frac {|\sigma(a_m(p)|^2}{p^s}
\leq C_{n,m}^2 \log\frac 1{s-1}+O(1),\quad \text{as $s\to 1^+$}.
$$
Let $N_K=\prod_{\sigma\in {\rm Gal}(K/\Bbb Q)} N_\sigma$.

By Assumption (2), 
there exists an integer $M>0$ so that 
$Ma_m(p)\in \mathcal{O}_K$ if $p\nmid N$. 
For $c>0$, consider two sets:
\begin{eqnarray*}
Y(c)&=&\{ \text{$a\in\mathcal O_K \,| \, |\sigma(a)|^2\leq c$ for any $\sigma\in {\rm Gal}(K/\Bbb Q)$}\},\\
X(c)&=&\{\text{$p\, |\,$  at least one of $Ma_m(p)$, $m=1,...,n$, does not belong to $Y(c)$, or $p| N_K$}\}.
\end{eqnarray*}

Note that since $\mathcal O_K$ is a lattice in $K$, $Y(c)$ is a finite set for any $c>0$. 
Hence the set $\{ (Ma_1(p),...,Ma_n(p))\, |\,\, p\notin X(c)\}$
is finite, and so the set $\{ (a_1(p),...,a_n(p)) \,|\, p\notin X(c)\}$ is finite. 

Let $r=[K:\Bbb Q]$. 
If $p\in X(c)$ and $p\nmid N_K$, there exists $m$ such that $|\sigma_m(Ma_m(p))|^2>c$ for some $\sigma_m\in {\rm Gal}(K/\Bbb Q)$. Hence
$$c\sum_{p\in X(c)} p^{-s}\leq \sum_{m=1}^n \sum_{\sigma}\sum_{p\nmid N_K} \frac {|\sigma(Ma_m(p))|^2}{p^s}+O(1)\leq \left(\sum_{m=1}^n C_{n,m}^2\right) rM^2\log\frac 1{s-1}+O(1), \quad \text{as $s\to 1^+$}.
$$
Therefore, den.sup$X(c)\leq \frac {rM^2}c \left(\sum_{m=1}^n C_{n,m}^2\right)$. Take $c$ such that 
$c\geq \frac {rM^2 \sum_{m=1}^n C_{n,m}^2}{\eta}$, and let $X_\eta=X(c)$.
\end{proof}

\section{Bounds for the orders of certain subgroups of $GL_n(\F_q)$}\label{finite-groups}

Fix a positive integer $n\ge 1$. 
Let $q$ be a power of a rational prime $p$ and $\F_q$ be the finite field with $q$ elements. 
Let $G$ be a subgroup of ${\rm GL}_n(\F_q)$. We say $G$ is semisimple if 
the natural action of $G$ on $V:=\F^{\oplus n}_q$ is semisimple or equivalently $V$ is a semisimple 
$G$-module. We say $G$ is an irreducible (resp. absolutely irreducible) subgroup of  ${GL}_n(\F_q)$ 
if $V$ (resp. $V\otimes_{\F_q}\overline{\F}_q$) is an irreducible $G$-module. 
As in \cite{d&s}, for positive constants $N$ and $\eta,\ (0<\eta<1)$, we introduce the following property $C(\eta,N)$ for $G$:
$$
C(\eta,N):\ \mbox{there exists a subset $H$ of $G$ such that}\ 
\left\{\begin{array}{ll}
(i)\  (1-\eta)|G| \le |H|, \\
(ii)\  |\{\det(1-hT)\in \F_q[T]|\ h\in H \}|\le N.
\end{array}\right. 
$$

\begin{thm}\label{lp}[Theorem 0.2 of \cite{lp}] There exists a constant $J_1(n)$, depending only on $n$ 
such that any finite subgroup $G$ of $GL_n(\F_q)$ possesses a series    
$G=G_0\supset G_1\supset G_2 \supset G_3$ of subgroups such that $G_{i}$ is normal in $G_{i-1}$ for each $1\le i\le 3$ and it satisfies   
\begin{enumerate}
\item $[G:G_1]\le J_1(n)$;
\item $G_1/G_2$ is a direct product of finite simple groups of Lie type in characteristic $p$  
and the number of direct factors is bounded uniformly in $n$;
\item $G_2/G_3$ is abelian of order not divisible by $p$; 
\item $G_3$ is a $p$-group.
\end{enumerate}
\end{thm}

\begin{remark} In the above theorem, the boundedness of the number of direct factors is not in the statement of Theorem 0.2 of \cite{lp}.
However, it is implicit in the proof of the theorem. It is important for our purpose.
\end{remark}

\begin{thm}\label{sup}[Chap. V, Section 19, Th. 7, of \cite{sup}] There exists a constant $J_2(n)$, depending only on $n$ 
such that any solvable subgroup $G$ of $GL_n(\F_q)$ possesses a normal subgroup $N$ such that 
\begin{enumerate}
\item $[G:N]\le J_2(n)$, 
\item  $N$ is conjugate to a subgroup of the group of upper triangular matrices in $GL_n(\F_q)$. 
\end{enumerate}
\end{thm}

\begin{cor}\label{trivial} Assume that $p>J_2(n)$. Let $G$ be any semisimple, solvable subgroup of $GL_n(\F_q)$. 
If $G'$ is a normal $p$-subgroup of $G$, then 
$G'$ is trivial.
\end{cor}
\begin{proof} By Theorem \ref{sup}, there exists a normal subgroup $N$ such that 
$[G:N]<p$. Hence the $p$-group $G'$ is a subgroup of $N$. 
By Theorem \ref{sup}, $N$ is conjugate to a subgroup of the group of upper triangular matrices 
and in particular we may assume that $G'$ is a subgroup consisting of unipotent,  upper triangular matrices.
By Clifford's theorem \cite{A}, p. 17, $G'$ is semisimple. Hence $G'=1$.
\end{proof}

\begin{cor}\label{trivial1} Assume that $p>J_2(n)$. Let $G, G_1, G_2, G_3$ be as in Theorem \ref{lp}. Then $G_3=1$.
\end{cor}
\begin{proof} Since any $p$-group is solvable, so is $G_2$. Then Corollary \ref{trivial} implies the assertion.
\end{proof}

\begin{thm}\label{Jordan}[cf. Theorem 0.1 of \cite{lp}] Let $k$ be an algebraically closed field of characteristic zero. For every $n$, there exists a constant $J_3(n)$, depending only on $n$ 
such that any finite subgroup $G$ of $GL_n(k)$ possesses an abelian normal subgroup $A$ such that 
$[G:A]\le J_3(n)$. 
\end{thm}

Henceforth we fix $n\ge 1$ and a positive number $C(n)$ so that 
$$C(n)>\max\{n+3, J_1(n),J_2(n),J_3(n)\}.$$
We will implicitly use the assumption $C(n)>n+3$ in the proof of 
Proposition \ref{structure} later to apply a result of \cite{gura}. (In Theorem B of \cite{gura}, one needs $p\geq n+3$, not $n-3$. It was pointed out by F. Herzig.)

The following lemma is easy to prove:
\begin{lem}\label{easy} Let $G$ be a subgroup of $GL_n(\F_q)$, and $G'$ a subgroup of $G$. Let $[G:G']<d$. 
Then if $G$ satisfies $C(\eta,N)$ for $(0<\eta<1/d)$, then $G'$ satisfies $C(d\eta,N)$.
\end{lem}
\begin{proof} Set $M=[G:G']$. Take a subset $H$ from the first property of $C(\eta,N)$ for $G$. Then one can see that 
$$|H|\ge (1-\eta)|G|=(M-M\eta)|G'|\ge (1-d\eta)|G'|$$
giving the claim. 
\end{proof}

\begin{prop}\label{irr1} Let $G$ be a semisimple subgroup of  ${GL}_n(\F_q)$ with the order not 
divisible by $p$. 
If $G$ satisfies  $C(\eta,N)$ for $(0<\eta<\frac{1}{C(n)})$, then $|G|\leq B$, where $B=B(\eta,N,n)$ is a constant 
depending only on $\eta,N$, and $n$.
\end{prop}
\begin{proof}
By assumption, we may assume that $G$ is a subgroup of $GL_n(\C)$.  
For example, we apply Schur-Zassenhaus' theorem (cf. \cite{DM}, page 829) to the natural projection 
$GL_n(W(\F_q))\lra GL_n(\F_q)$ where $W(\F_q)$ the ring of Witt vectors and then get a lift $G$ to $GL_n(W(\F_q))$. 
Then we have only to compose this with an embedding $GL_n(W(\F_q))\hookrightarrow GL_n(\C)$. 

By Theorem \ref{Jordan} there exists an abelian normal 
subgroup $A$ of $G$ such that $[G:A]\leq J_3(n)$. By Lemma \ref{easy}, 
$A$ satisfies $C(C(n)\eta,N)$. We may assume that $A$ is a subgroup consisting of diagonal matrices in $GL_n(\F_q)$.  Then one has 
$$(1-C(n)\eta)|A|\le |H| \le n!N$$
giving a bound of $|A|$. Since $[G:A]$ is bounded, so is $|G|=[G:A]\times|A|$.  
\end{proof}

\begin{prop}\label{irr2} Let $G$ be an irreducible subgroup of  ${GL}_n(\F_q)$. 
Then the following properties hold:
\begin{enumerate}
\item there exists a finite extension field $\F_{q^r}$ and an absolutely irreducible subgroup 
$G'\subset GL_m(\F_{q^r})$ with $n=rm$ such that $G$ is isomorphic to $G'$. Furthermore, for any $g\in G$ and the corresponding $g'\in G'$ under this isomorphism, 
$$f_g(T)=\ds\prod_{\sigma\in {\rm Gal}(\F_{q^r}/\F_q)}f_{g'}(T)^\sigma,
$$ 
where $f_g(T),f_{g'}(T)$ stand for 
characteristic polynomials of $g, g'$, resp. 
\item the center of $G'$ is a cyclic subgroup of $\F^\times_{q^r}I_m\subset  GL_m(\F_{q^r})$.
\end{enumerate}
\end{prop}

\begin{proof} 
By Schur's lemma, the centralizer $Z=Z_{M_n(\F_q)}(G)$ is a finite division ring. 
By Wedderburn's theorem, $Z$ is a finite field over $\F_q$, say $\F_{q^r}$, since $Z$ contains $\F_q I_n$. 
Since $Z^\times$ acts on $V:=\F^{\oplus n}_q$ faithfully, dim$_{\F_q}\, \F_{q^r}=r$ has to divide $n$. Put $m=\frac{n}{r}$. 
We view $V$ as a $\F_{q^r}$-module. Then one has a faithful representation $G\lra GL(V)\simeq GL_m(\F_{q^r})$. 
To be more precise, if we take a basis $\{e_1,\ldots,e_m\}$ of $V$ as a $\F_{q^r}$-module and a generator $\alpha\in 
\F_{q^r}$ over $\F_q$, then 
a basis of $V$ is given by $\{\alpha^{p^i}e_j\ |\ 0\le i\le r-1,\ 1\le j \le m  \}$. 
Denote by $G'$ the image of  $G$ under this representation. Then $G'$ is absolutely irreducible since 
$\F_{q^r}=Z=Z_{M_n(\F_{q})}(G)=Z_{M_m(\F_{q^r})}(G')$.  
The last claim follows from the direct calculation in this explicit basis.

For the claim (2), let $g$ be an element in the center $Z(G')$. Since $g$ commutes with the action of $G'$, 
it belongs to $Z_{M_m(\F_{q^r})}(G')^\times=\F^\times_{q^r}$. 
\end{proof}

Let $T$ be the group of all diagonal matrices in ${\rm GL}_n(\F_q)$. 
Let $G$ be a semisimple subgroup of  ${\rm GL}_n(\F_q)$. Assume $p>C(n)$. 
Fix a series of normal subgroups $G\supset G_1\supset G_2 \supset G_3$ in Theorem \ref{lp} for $G$. By Corollary \ref{trivial1}, $G_3=1$.
Assume $G_1\neq G_2$ until the end of the proof of the following proposition. 

By Clifford's theorem, $V=\F^{\oplus n}_q$ is a semisimple $G_1$-module. Therefore 
we have a decomposition $V=\bigoplus_{1\le i \le m} W_i$ into irreducible components as a $G_1$-module, where dim$_{\F_q}\, W_i=n_i r_i$ for each $i$.
By Proposition \ref{irr2}, we may assume that each $(G_1,W_i)$ is an absolutely irreducible module 
over the field extension $\F_{q^{r_i}}$ of $\F_q$, and we have a faithful representation 
$\pi_i: G_1\lra GL(W_i)\simeq GL_{n_i}(\F_{q^{r_i}})$. 
We denote by $G^{(i)}$ the image of $G_1$ under $\pi_i$. Then we get an injection
$$G_1\hookrightarrow \prod_{i=1}^m G^{(i)}$$
which is not necessarily surjective, but we will see this map will be an isomorphism under the natural 
quotients.  
Note that clearly $r_i <n$. 

\begin{prop}\label{structure} Under the above setting, the following properties hold:   

\begin{enumerate}
\item for each $i=1,...,m$, the center $Z(G^{(i)})$ is a subgroup of $\F^\times_{q^{r_i}} {\rm id}_{W_i}$, and 
$G_2$ is a subgroup of $T$ so that $\pi_i(G_2)\subset Z(G^{(i)})$. Further 
$G_2\hookrightarrow\ds\prod_{1\le i\le m}\pi_i(G_2)$. 
\item For each $i=1,...,m$, there exists a simple and simply connected linear algebraic group $\mathcal{G}_i$ over $\F_{q^{r_i}}$ realized inside $GL(W_i)$ such that 
$G^{(i)}=Z(G^{(i)})\mathcal{G}_i(\F_{q^{r_i}})$ and $Z(\mathcal{G}_i(\F_{q^{r_i}})) \subset Z(G^{(i)})$. 
In particular, the natural map 
$$G_1/G_2\lra \ds\prod_{1\le i\le m}G^{(i)}/Z(G^{(i)})\simeq \ds\prod_{1\le i\le m}
\mathcal{G}_i(\F_{q^{r_i}})/Z(\mathcal{G}_i(\F_{q^{r_i}}))$$
is isomorphic 
where each component of the right hand side is a simple Chevalley group (cf. \cite{mt}). 
Further  $G_1\hookrightarrow \ds\prod_{1\le i\le m}Z(G^{(i)})\mathcal{G}_i(\F_{q^{r_i}})$ 
with respect to the decomposition $V=\bigoplus_{1\le i \le m} W_i$, and there exist a constant $C_1(n)$ depending only on $n$ so that 
$$|G_1|\ge C_1(n)\ds\prod_{1\le i\le m}|\mathcal{G}_i(\F_{q^{r_i}})|.$$ 
\end{enumerate}
\end{prop}
\begin{proof}
The first part of (1) follows from Proposition \ref{irr2}-(2). 
Since $(G_1,W_i)$ is (absolutely) irreducible, by Clifford's theorem, $W_i|_{G_2}$ decomposes into 
isotypical representations of $1$-dimensional representations. Hence $\pi_i(G_2)$ are scalar matrices. 
Hence it clearly commutes with $G^{(i)}$. The latter claim is clear from the injectivity of 
$G_1\hookrightarrow \ds\prod_{i=1}^m G^{(i)}$. 
 
We now prove the second claim. 
Let $\Gamma^0_i$ be the group generated by all elements of $p$-th power order in $G^{(i)}$. 
Then by Theorem B of \cite{gura} (see also step 1 in the proof of Proposition A.7 of \cite{ghtt}), 
$\Gamma^0_i/Z(\Gamma^0_i)$ is a simple Chevalley group. 
Since $\Gamma^0_i$ is a normal subgroup of $G^{(i)}$, so is  $Z(G^{(i)})\cdot \Gamma^0_i/Z(G^{(i)})$ 
in $G^{(i)}/Z(G^{(i)})$. However $G^{(i)}/Z(G^{(i)})$ is by construction (see Theorem \ref{lp}-(b)), a 
simple group. Then one has $Z(G^{(i)})\cdot \Gamma^0_i/Z(G^{(i)})=G^{(i)}/Z(G^{(i)})$.
Hence $Z(G^{(i)})\Gamma^0_i=G^{(i)}$. The surjective map 
$\Gamma^0_i\lra G^{(i)}/Z(G^{(i)})$ induces an isomorphism 
$\Gamma^0_i/Z(\Gamma^0_i)\stackrel{\sim}{\lra} G^{(i)}/Z(G^{(i)})$. If an element $g\in Z(\Gamma^0_i)$ does not 
belong to $ Z(G^{(i)})$, then the previous isomorphism never be isomorphic, hence it gives a contradiction. Hence one has $Z(\Gamma^0_i)\subset Z(G^{(i)})\subset \F^\times_{q^{r_i}}{\rm id}_{W_i}$. 
This means that the image of $\Gamma^0_i$ under the projective map ${\rm GL}(W_i)\lra {\rm PGL}(W_i)$ is 
$\Gamma^0_i/Z(\Gamma^0_i)$ and it is a simple Chevalley group. Then the claim follows 
by looking for any simple Chevalley group which appears in this way (see \cite{mt}). Therefore there exists a 
simple and simply connected algebraic group $\mathcal{G}_i$ over $\F_{q^{r_i}}$ such that 
$G^{(i)}=Z(G^{(i)})\Gamma^0_i=Z(G^{(i)})\mathcal{G}_i(\F_{q^{r_i}})$. 

To prove the last claim, let us consider the following commutative diagram 
$$
\begin{CD}
G_1 @>  >> \ds\prod_{1\le i\le m}Z(G^{(i)})\mathcal{G}_i(\F_{q^{r_i}})\\
\pi_1@VVV \pi_2@VVV  \\
G_1/G_2 @>\sim >> \ds\prod_{1\le i\le m}
\mathcal{G}_i(\F_{q^{r_i}})/Z(\mathcal{G}_i(\F_{q^{r_i}})).
\end{CD}
$$
where the top arrow is an injective map which is defined by the decomposition $V=\ds\bigoplus_{1\le i \le m} W_i$. 
Here $\pi_1$ and $\pi_2$ stand for natural projections. 
Then one has 
$$|G_1|\ge  C_1(n)\ds\prod_{1\le i\le m}
|\mathcal{G}_i(\F_{q^{r_i}})|,\ C_1(n):=\frac{1}{\ds\prod_{1\le i\le m}|Z(\mathcal{G}_i(\F_{q^{r_i}}))|}.$$
Since $\mathcal{G}_i$ is simple and simply connected algebraic group, the cardinality of 
$Z(\mathcal{G}_i(\F_{q^{r_i}}))$ is depending only on 
the rank of $\mathcal{G}_i$, hence on $n$. This completes the proof.  
\end{proof}

For any subgroup $D\subset GL_n(\F_q)$ and each $g\in D$, we define by $M_D(g)$ the number of elements of $D$ which have the same characteristic polynomial as $g$ and put $M_D=\ds\max_{g\in D} \{M_D(g)\}$. Note that 
for any subgroup $A$ of the center $\F^\times_q I_n$, $M_{A\cdot D}=M_D$ and $M_{D_1}\le M_{D_2}$ for 
subgroups $D_1\subset D_2\subset GL_n(\F_q)$. This simple observations 
will be used in the proof of Theorem \ref{main1} below.

\begin{lem}\label{simple-lemma} For each $i$ $(1\le i \le m)$, let $D_i$ be a subgroup of $GL_{n_i}(\F_{q^{r_i}})$. 
We identify the product $D:=\ds\prod_{1\le i\le m} D_i$ with the Levi subgroup of the parabolic subgroup $P_{(n_1,\ldots,n_m)}$ in 
$GL_n(\overline{\F}_q)$ with respect to the partition $n=n_1+\cdots+n_m$. 
Then there exists a constant $C_2(n)$ depending only on $n$ so that 
$$M_D \le C_2(n)\ds\prod_{1\le i\le m} M_{D_i}.
$$   
\end{lem}
\begin{proof} We will give a very rough estimation for $C_2(n)$. 
For any $g=(g_1,\ldots,g_m)$, $f_g(T)=\ds\prod_{1\le i\le m}f_{g_i}(T)$. We denote the eigenvalues by 
$\alpha^{(1)}_1,\ldots,\alpha^{(1)}_{n_1},
\alpha^{(2)}_1,\ldots,\alpha^{(2)}_{n_2},\ldots, 
\alpha^{(m)}_1,\ldots,\alpha^{(m)}_{n_m}$ which are not necessarily different from each other. 
Then the number of all permutations which preserve the type $(n_1,\ldots,n_m)$ is $\ds\frac{n!}{n_1!\cdots n_m!}$. 
We may take this as $C_2(n)$. 
\end{proof}

\begin{prop}\label{finite-group} [Proposition 3.1 of \cite{lp} or Lemma 3.5 of \cite{N}]
For any connected algebraic group $G$ over $\F_q$, we have
$$(\sqrt{q}-1)^{2 \text{dim}\, G}\leq |G(\F_q)|\leq (\sqrt{q}+1)^{2 \text{dim}\, G}.
$$
For connected linear algebraic groups, one has a stronger estimate 
$$(q-1)^{\text{dim}\, G}\leq |G(\F_q)|\leq (q+1)^{\text{dim}\, G}.
$$
\end{prop}

Let $G$ be a simple and simply connected algebraic group over a finite field $\F_q$. 
We follow \cite{C} for the following proposition.
\begin{prop}\label{main} Let $l=rank(G)$, and let $A$ be a semisimple element in $G$ and $C(A)$ be the centralizer of $A$ in $G(\Bbb F_q)$, and $d=\text{dim}\, C(A)$. Then  
$$\frac {q^d}{(q+1)^d} \frac {|G(\Bbb F_q)|}{q^l}\leq M_G(A) \leq \frac {q^d}{(q-1)^d} \frac {|G(\Bbb F_q)|}{q^l}.
$$
\end{prop}
\begin{proof} Let $\Delta_G(A)$ be the set of $g\in G$ which has the same characteristic polynomial as $A$ so that 
$M_G(A)=|\Delta_G(A)|$. 
Suppose $g\in \Delta_G(A)$. Then $g=g_s g_u$ with $g_s$ semisimple, $g_u$ unipotent. Then
$det(1-T g_s)=det(1-T A)$.

Since $g_s$ and $A$ are conjugate in $G(\overline{\Bbb F}_q)$, they are conjugate in $G(\Bbb F_q)$ \cite{St1}. 
Over $\overline{\Bbb F}_q$, the algebraic group $C(A)$ is 
the centralizer in $G(\overline{\Bbb F}_p)$ of $A$. Since $A$ is semisimple and $G$ is simply connected, $C(A)$ is a connected reductive group \cite{St}. 
Since $C(A)$ contains any maximal torus of $G(\Bbb F_p)$ and $C(A)\subset G(\Bbb F_p)$, $rank(C(A))=l$.
By Steinberg \cite{St}, 
$$\#\{ \text{unipotent elements in $C(A)(\Bbb F_q)$}\}=q^{d-l}.
$$ 
Therefore,

\begin{eqnarray*} M_G(A) &=& \#\{ \text{pairs $(g_s, g_u)$} |\, \text{$g_s$ is $G(\Bbb F_p)$-conjugate to $A$ and $g_u\in (C(g_s))_u(\Bbb F_q)$ } \} \\
&=& q^{d-l} \#\{ \text{$g_s$ which is $G(\Bbb F_q)$-conjugate to $A$} \}= q^{d-l} \frac {\# G(\Bbb F_q)}{\# C(A)(\Bbb F_q)}.
\end{eqnarray*}

Since $C(A)$ is connected, by Proposition \ref{finite-group}, $(q-1)^d\leq \# C(A)(\Bbb F_p)\leq (q+1)^d$. Hence
our assertion follows.
\end{proof}

Here $M_G(A) \leq K\frac {|G(\Bbb F_q)|}{q^l}$ for a constant $K$ depending only on $dim\,G$.

\begin{thm}\label{main1}
Let $G$ be a semisimple subgroup of $GL_n(\F_q)$. Assume that $p>C(n)$. If $G$ satisfies the property $C(\eta,N)$ for $0<\eta<\frac{1}{C(n)}$, 
then either $|G|$ or $q$ is bounded by a constant depending only on $n$. 
Hence there exists a constant $B=B(\eta,N,n)$ depending only on $\eta,N,n$ such that $|G|\leq B$.
\end{thm}
\begin{proof} Take a series of normal subgroups $G\supset G_1\supset G_2 \supset G_3$ in Theorem \ref{lp}. 
By Corollary \ref{trivial1}, $G_3=\{1\}$. If $G_1=G_2$, then the claim follows from Proposition \ref{irr1}. 

Henceforth we assume $G_1\not= G_2$. Then by Proposition \ref{structure}, there exists a injective map 
$$G_1\lra \ds\prod_{1\le i\le m}Z(G^{(i)})\mathcal{G}_i(\F_{q^{r_i}}),\quad \mathcal{G}_i(\F_{q^{r_i}})\subset 
GL_{n_i}(\F_{q^{r_i}}), \quad 1\le i \le m
.$$   
Then one has  
$$M_{G_1}\le M_{D},\ D:=\ds\prod_{1\le i\le m}Z(G^{(i)})\mathcal{G}_i(\F_{q^{r_i}}).$$
Since $G$ satisfies $C(\eta,N)$, by Lemma \ref{easy}, $G_1$ satisfies $C( C(n)\eta,N)$. 
This means that 
$$(1-C(n)\eta)|G_1|\leq |H|.
$$
Then applying Proposition \ref{main} to $D=\ds\prod_{1\le i\le m}Z(G^{(i)})\mathcal{G}_i(\F_{q^{r_i}})$, and 
by Lemma \ref{simple-lemma}, one has 

\begin{eqnarray*}
&{}& (1-C(n)\eta)C_1(n)\ds\prod_{1\le i\le m}|\mathcal{G}_i(\F_{q^{r_i}})| \le   (1-C(n)\eta)|G_1| \leq |H|\le NM_{G_1} \\
&{}& \le N M_D \le  N C_2(n)\ds\prod_{1\le i \le m} M_{Z(G^{(i)})\mathcal{G}_i(\F_{q^{r_i}})} 
=N C_2(n)\ds\prod_{1\le i \le m} M_{\mathcal{G}_i(\F_{q^{r_i}})}  \\
&{}&\le  N K(n) C_2(n)\ds\prod_{1\le i \le m}\frac{|\mathcal{G}_i(\F_{q^{r_i}})|}{q^{r_i l_i}}
\end{eqnarray*}
with a constant $K(n)$, where $l_i={\rm rank}\, \mathcal{G}_i$. This gives us the bound 
$$q\le \ds\prod_{1\le i\le m}q^{r_i l_i}\le \frac{N K(n) C_2(n)}{(1-C(n)\eta)C_1(n)}.
$$
Hence the claim follows. 
\end{proof}
\begin{cor}\label{finiteness}
Let $S$ be an infinite set of rational primes. Suppose for each prime $\ell\in S$, the image of 
a mod $\ell$ semisimple Galois representation $\rho_\ell:G_\Q\lra GL_n(\F_\ell)$ satisfies $C(\eta,N)$ for 
$0<\eta<\ds\frac{1}{C(n)}$. Then there exists a constant $A=A(\eta,N,n)$ such that 
$|{\rm Im}\, \rho_\ell|\leq A$.  
\end{cor}

\section{Mod $\ell$ representations}

We state the assumption on the existence of mod $\ell$ representations.
\begin{conjecture}\label{mod-p-galois}
Let $\pi$ be a cuspidal representation of $GL_n(\A_\Q)$ satisfying Assumptions (1) and (2) in the introduction, namely,
 the finiteness of the Hecke field $\Q_\pi$, and the integrability of the Hecke polynomial $H_p(T)$. 
 
Then for all but finitely many $\ell$ coprime to $N$ and each finite place $\lambda$ of $\Q_\pi$ above $\ell$ with the 
residue field $\F_\lambda$, there exists a continuous semi-simple representation 
$$\rho_\lambda:G_\Q\lra GL_n(\bF_\lambda)
$$
which is unramified outside of $\ell N$, so that 
$$\det(I_n-\rho_\lambda({\rm Frob}_p) T)\equiv H_p(T)\, {\rm mod}\ \lambda,
$$
for any $p\nmid \ell N$.
\end{conjecture}

\section{Artin representations associated to cuspidal representations}

In this section we give a proof of the main theorem (Theorem \ref{artin-rep}). 
Let $\pi$ satisfy Assumptions (1)-(5) in the introduction. Let $\Sigma_\pi$ be the set of primes $\ell$ which are 
excluded in the statement of Conjecture \ref{mod-p-galois} 
for the existence of mod $\ell$ representation for $\pi$. 
We denote by $S_\pi$ the union of $\Sigma_\pi$ and the set of rational primes consisting of primes $p$ so that $\pi_p$ is ramified. 
Let $K$ be a Galois closure of $\Q_\pi$. By assumptions on $\pi$, this is a finite extension of $\Q$.    
Let $P_K$ be the set prime numbers $\ell$ which splits completely in $K$. 
For each $\ell\in P_K$, choose a finite place $\lambda_\ell$ of $K$ dividing $\ell$. 
By Conjecture \ref{mod-p-galois}, there exists a continuous semi-simple representation 
$$\rho_\ell:G_\Q\lra GL_n(\overline{\F}_{\ell})$$
which is unramified outside $S_\pi\cup\{\ell\}$, and 
$$\det(I_n-\rho_\ell({\rm Frob}_p)T)\equiv H_p(T) \, {\rm mod}\ \lambda_\ell.
$$
By Lemma 6.13 of \cite{d&s}, we may assume that the image of $\rho_\ell$ takes the values in 
$GL_n(\F_\ell)$.  
Let $G_\ell:={\rm Im}\, \rho_\ell$. 

\begin{lem}\label{condition}For any $\eta,\, 0<\eta<1$, there exists a constant $M$ such that 
$G_\ell$ satisfies $C(\eta,M)$ for every $\ell\in P_K$.
\end{lem}
\begin{proof}
By Proposition \ref{density}, if we let $\mathcal{M}:=\{ H_p(T)\, |\ p\not\in X_\eta\},$
then $\mathcal{M}$ is a finite set. Let $M:=|\mathcal{M}|$ which will be a desired constant as below. 
Let us consider the subset of $G_\ell$ defined by 
$$H_\ell:=\{g\in G_\ell\ |\ g \stackrel{\tiny{G_\ell}}{\sim} \rho_\ell({\rm Frob}_p) {\rm\ for\ some\ }p\not\in X_\eta  \}.$$
By Chebotarev density theorem, one has 
$$1=\frac{|H_\ell|}{|G_\ell|}+{\rm den}(X_\eta)\le \frac{|H_\ell|}{|G_\ell|}+{\rm den.sup}(X_\eta)
\le \frac{|H_\ell|}{|G_\ell|}+\eta,$$
giving $(1-\eta)|G_\ell|\le |H_\ell|$. 

The characteristic polynomial of each element of $H_\ell$ is the reduction of some element of $\mathcal{M}$. Therefore one has 
$$|\{\det(I_n-hT)\ |\ h\in H_\ell  \}|\le M.$$
\end{proof}

By Lemma \ref{condition} together with Corollary \ref{finiteness},  
there exists a constant $A$ such that $|G_\ell|\le A$ for any $\ell\in P_K$. 
Let $Y$ be the set of polynomials $\prod_{i=1}^n (1-\alpha_i T)$, where $\alpha_i$'s are 
roots of unity of order less than $A$. 

If $p\not\in S_\pi$, for all $\ell\in P_K$ with $\ell\neq p$, there exists $R(T)\in Y$ such that 
$$H_p(T)\equiv R(T) \ {\rm mod}\ \lambda_\ell.
$$
Since $Y$ is finite and $P_L$ is infinite, 
$$H_p(T)=R(T).
$$

Let $P_K'$ be the set of $\ell\in P_K$ such that $\ell>A$ and for $R,S\in Y$ with $R\not=S$, $R\not\equiv S\ {\rm mod}\ \lambda_\ell$. 
Then it is easy to see that $P_K'$ is infinite. 
For each $\ell\in P_K'$, $\ell$ does not divide $|G_\ell|$, since $\ell>A\ge |G_\ell|$. Let 
$\pi_\ell:GL_n(\mathcal{O}_{\lambda_\ell})\lra GL_n(\F_\ell)$ be the reduction modulo $\lambda$.
Applying a profinite version of Schur-Zassenhaus' theorem (cf. \cite{RZ}, page 40, Theorem 2.3.15) to $\pi^{-1}(G_\ell)$ 
and $\pi^{-1}(G_\ell)\cap {\rm Ker}(\pi)$ (note that the latter group is a Hall subgroup of $\pi^{-1}(G_\ell)$ 
in the sense of \cite{RZ}),   
there exists a subgroup $H\subset \pi^{-1}(G_\ell)$ such that $\pi^{-1}(G_\ell)=H\cdot(\pi^{-1}(G_\ell)\cap {\rm Ker}(\pi))$ 
and $H\cap(\pi^{-1}(G_\ell)\cap {\rm Ker}(\pi))={1}$. 
Then the composition of the inclusion $H\hookrightarrow \pi^{-1}(G_\ell)$ and 
$\pi$ induces an isomorphism  
$$H \stackrel{\sim}{\lra} G_\ell={\rm Im}\: \rho_\ell.
$$ 
Hence we have a lift $\rho'_\ell:G_\Q\lra GL_n(\mathcal{O}_{\lambda_\ell})$ of $\rho_\ell$.
Since the coefficient of $\rho'_\ell$ is of characteristic zero and its image is finite, for $p\nmid N\ell$, one has $\det(I_n-\rho'_\ell({\rm Frob}_p)T)\in Y$. 
On the other hand, we have 
$$\det(I_n-\rho'_\ell({\rm Frob}_p)T)\equiv H_p(T) {\rm mod}\ \lambda_\ell.
$$
Since $\ell\in P_K'$, the above congruence relation implies the equality 
$$\det(I_n-\rho'_\ell({\rm Frob}_p)T)=H_p(T).
$$  
for all $p\nmid N\ell$. Now we replace $\ell$ with another prime $\ell'\in P_K'$. Then one has 
$\rho'_{\ell'}:G_\Q\lra GL_n(\mathcal{O}_{\lambda_{\ell'}})$ such that 
$$\det(I_n-\rho'_\ell({\rm Frob}_p)T)=\det(I_n-\rho'_{\ell'}({\rm Frob}_p)T)
$$
for all $p\nmid N\ell\ell'$. 
By Chebotarev density theorem, one has $\rho'_{\ell'}\sim \rho'_{\ell}$ and this means that 
$\rho'_{\ell}$ is unramified at $\ell$. Hence we have the desired representation 
\begin{equation*}\label{artin-rep-eq}
\rho_\pi:=\rho'_{\ell}:G_\Q\lra GL_n(\mathcal{O}_{\lambda_\ell})\hookrightarrow GL_n(\C),
\end{equation*}
where the second map comes from a fixed embedding $\mathcal{O}_{\lambda_\ell}\hookrightarrow \C$. 
This representation is independent of any choice of such a embedding by Chebotarev density theorem. 
The infinity type $\pi_\infty$ was determined in Proposition \ref{infinity}.

\begin{cor} Let $\pi$ be a cuspidal representation of $GL_n(\A_\Q)$ which satisfies Assumptions (1)-(5) in the introduction.
Then $\pi_p$ is tempered for all $p$.
\end{cor}
\begin{proof} By Theorem \ref{artin-rep}, there exists the Artin representation $\rho_\pi : G_\Q\lra GL_n(\C)$ such that
for almost all $q$,

$$\det(I_n-\rho_\pi({\rm Frob}_q)T)=H_q(T).
$$ 
This shows that $\pi_q$ is tempered for almost all $q$. 

Suppose $\pi_p$ is non-tempered. We apply Proposition A.1 of \cite{martin} to the Rankin-Selberg $L$-function $L(s,\pi\times\tilde\pi)$:
Since $L(s,\pi_q\times\tilde\pi_q)=L(s,\rho_q\times\tilde\rho_q)$ for almost all $q$, in particular, we have 
$L(s,\pi_p\times \tilde\pi_p)=L(s,\rho_p\times \tilde\rho_p)$. Suppose $\pi_p$ is of the form (\ref{rama}). Then
by (\ref{L-fun}), the left hand side has a factor $L(s-2r_1,\eta_1\times\tilde\eta_1)$ which has a pole at $s=2r_1>0$. On the other hand, the right hand side is holomorphic for $Re(s)>0$. Contradiction. 
Hence $\pi_p$ is tempered for all $p$.
\end{proof}

\section{Non-holomorphic Siegel modular forms and holomorphic Siegel modular forms via the congruence method}

In this section we follow the notation of \cite{kim-yam}. 
Let us first recall the existence of a Galois representation for any holomorphic 
Siegel modular form of weight $(k_1,k_2)$, $k_1\geq k_2\geq 2$, for $GSp_4$. 
Thanks to the works of \cite{laumon}, \cite{wei} and \cite{taylor}
with the classification of CAP forms (\cite{ps}, \cite{soudry}, \cite{schmidt}) and 
endoscopic representations for $GSp_4$ (\cite{roberts}), we can associate a Galois representation to $F$.  

\begin{thm}\label{galois}For any prime $\ell$,  
there exists a number field $E$ including $\Q_F$, such that for each rational prime $\ell$ and a finite place 
$\lambda|\ell$ of $\Q_F$, there exists a continuous representation 
$\rho_{F,\ell}:G_\Q\lra GL_4(E_\lambda)$
which is unramified outside of $\ell N$ so that 
$$\det(I_4-\rho_{F,\ell}({\rm Frob}_p)p^{-s})^{-1}=L_p(s,F)=L_p(s-\tfrac{k_1+k_2-3}{2},\pi_F)$$
for any $p\nmid \ell N$. Furthermore, if $k_1\ge k_2\ge 3$ and $\pi_F$ is neither endoscopic nor CAP, then 
the image of $\rho_{F,\ell}$ can be taken in $GSp_4(E_\lambda)$.  
\end{thm}
Let us denote by $\brho_{F,\ell}:G_\Q\lra GL_4(\F)$ the reduction modulo $\lambda (\lambda|\ell)$ where 
$\F$ is the residue field of $\lambda$.  

Let $f$ be an elliptic newform of weight one which is neither of dihedral nor of tetrahedral type. Then this gives rise to a unique 
Artin representation $\rho_f:G_\Q\lra GL_2(\C)$. 
Since the image is finite, we can take a finite extension $K$ of $\Q$ so that 
${\rm Im}(\rho_f)\subset GL_2(\mathcal{O})$ where $\mathcal{O}$ is the ring of integers of $K$. 
Then taking the reduction modulo a prime ideal above 
a rational prime $\ell$, we obtain a mod $\ell$ representation 
$\brho_{f,\ell}:G_\Q\lra GL_2(\F)$. 

By Theorem 10.1 of \cite{kim-yam}, there exists 
a real analytic Siegel modular form $F$ of weight $(2,1)$ with 
eigenvalues $-\ds\frac{5}{12}$ (resp. 0) for $\Delta_1$ (resp. $\Delta_2$) (see \cite{kim-yam} for $\Delta_i$) 
such that $\pi_F\sim {\rm Sym}^3(\pi_f)$. By using this, we obtain a mod $\ell$ representation 
$\pi_{F,\ell}:G_\Q\lra GL_4(\F)$ for $F$. 

On the other hand, by multiplying Hasse invariant of weight $\ell-1$, we obtain  
an eigenform $g$ of weight $1+a(p-1)$ for any positive integer $a$ such that 
$g$ is congruent to $f$ modulo $\ell$ hence $\brho_{g,\ell}\simeq \brho_{f,\ell}$. 
By using symmetric cube lift and generic transfer from $GSp_4$ to $GL_4$, one can show the existence of 
a holomorphic Siegel cusp form G of weight $(k_1,k_2)=(2a(\ell-1)+2,a(\ell-1)+1)$ such that 
$\rho_{G,\ell}={\rm Sym}^3(\rho_{g,\ell})$. From this one concludes that there exist  
a non-holomorphic Siegel modular form $F$ of weight $(2,1)$ and 
a holomorphic Siegel modular form $G$ of weight $(2a(\ell-1)+2,a(\ell-1)+1)$ such that 
$$\brho_{F,\ell}\simeq \brho_{G,\ell}.$$ 
We denote by this property $F\equiv G \mod \ell$ provided if the existence of and mod $\ell$ representations of $F$ and $G$ is guaranteed.   

We can also construct such $F$ and $G$ by using endoscopic lift from a pair $(f_1,f_2), \pi_{f_1}\not\simeq \pi_{f_2}$ of 
elliptic newform of weight one whose central characters are same as follows. 
By using theta lift  (cf \cite{roberts}) and Section 5 of \cite{kim-yam}, there exists a 
a real analytic Siegel modular form of weight $(2,1)$ as above such that 
$\rho_{F,\ell}\simeq \rho_{f_1}\oplus \rho_{f_2}$. 
By multiplying Hasse invariant again, one has a pair of 
elliptic modular forms  $(g_1,g_2),\pi_{g_1}\not\simeq \pi_{g_2}$ of 
elliptic newform $g_1$ (resp. $g_2$) of weight $r_1=1+a(\ell-1)$ (resp. $r_2=1+b(\ell-1)$) 
with the same central character. Then by using theta lift (cf \cite{roberts}), 
one can construct a holomorphic Siegel cusp form $G$ of weight 
$(k_1,k_2)=(\frac{(\ell-1)(a+b)}{2}+1,\frac{(\ell-1)(a-b)}{2}+2)$ such that 
$\rho_{F,\ell}\simeq \rho_{g_1,\ell}\oplus  \rho_{g_2,\ell}$. 
Taking reduction modulo $\ell$, one concludes $F\equiv G \mod \ell $. 

For such $F$ (and $G$), the mod $\ell$ representation $\brho_{F,\ell}$ has a remarkable property 
that $\brho_{F,\ell}$ is unramified at $\ell$. In case elliptic newform, this property characterizes 
a weight $\ell$ form so that it comes from a weight one form by multiplying Hasse invariant of 
weight $\ell-1$.  This principle is discussed in Proposition 2.7 of \cite{edix} which plays an important role for proving 
Serre conjecture. So this gives rise to the 
following natural question: 
\begin{ques}\label{mazur}Let $G$ be a holomorphic Siegel cusp form of weight $(k_1,k_2)$ so that 
$k_1-1$ and $k_2-2$ are both divided by $\ell-1$, where $\ell$ is a rational prime. 
Assume that $\brho_{G,\ell}$ is unramified at $\ell$. 
Can one associate a non-holomorphic Siegel cusp form $F$ of weight $(2,1)$ with 
a mod $\ell$ representation such that $F\equiv G\mod \ell$?  
\end{ques}

\section{Supplement to our paper \cite{kim-yam}}

In \cite{kim-yam}, we used Arthur's conjectural result on the correspondence between cuspidal representations of 
$GSp_4$ and $GL_4$ \cite{arthur}. It depends on the stabilization of the trace formula, which is not proved yet. In this section, we explain how to get around this by using the transfer from $Sp_4$ to $GL_5$ in \cite{arthur1}. The result depends on the twisted fundamental lemma which may have been resolved by now. 

Let $\pi=\pi_F$ be the cuspidal representation of $GSp_4(\A_\Q)$ attached to the Siegel cusp form $F$ of weight $(2,1)$. We showed in \cite{kim-yam} that $\pi_F$ is not a CAP representation.
Let $\pi'$ be one of components of $\pi|_{Sp_4(\Bbb A)}$. Then it is a cuspidal representation of $Sp_4(\A_\Q)$.
By \cite{arthur1}, $\pi'$ corresponds to an automorphic representation $\Pi_5$ of $GL_5$. Since $\pi'$ is not a CAP representation, 
$\Pi_5$ is either cuspidal or an isobaric representation.

By using the descent construction \cite{GRS}, we can find a globally generic cuspidal representation $\tau'$ of $Sp_4(\A_\Q)$ which is in the same $L$-packet as $\pi'$. Now let $\tau$ be a globally generic cuspidal representation of $GSp_4(\A_\Q)$ such that $\tau'$ occurs in the restriction $\tau|_{Sp_4(\Bbb A)}$. By \cite{asgari&shahidi}, we have a functorial lift $\Pi$ of $\tau$ as an automorphic representation of $GL_4$.
This $\Pi$ is the transfer of $\pi$. 
We can see easily that $\wedge^2(\Pi)=\Pi_5\otimes\omega_{\pi}\boxplus \omega_{\pi}$, i.e.,
$\Pi_5$ is the transfer of $\pi$ to $GL_5$ corresponding to the $L$-group homomorphism $GSp_4(\Bbb C)\longrightarrow GL_5(\C)$. Hence we do not need the exterior square lift of $\Pi$ in \cite{kim1} in order to obtain $\Pi_5$.

\section{Non-holomorphic Siegel cusp forms of weight $(2,1)$ attached to cusp forms on imaginary quadratic fields}\label{imag}

In this section, as a supplement to our paper \cite{kim-yam}, we use the idea of \cite{BR} to construct a non-holomorphic Siegel cusp form of weight $(2,1)$ attached to Maass forms for $GL_2/\Q$ and 
cuspidal representations of $GL_2$ over imaginary quadratic fields. This idea was used by
Harris, Soudry, and Taylor \cite{HST} to construct holomorphic Siegel cusp forms from certain modular forms over imaginary quadratic fields. 

Let $K=\Bbb Q[\sqrt{-D}]$ be an imaginary quadratic field.
Let ${\rm Gal}(K/\Q)=\{1,\theta\}$, and $\omega_{K/\Q}$ be the quadratic character attached to $K/\Q$ i.e., 
$\omega_{K/\Q}(p)=(\frac {-D}p)$.

Let $\bold G=R_{K/\Q} GL_2$ be the quasi-split group obtained by the restriction of scalars. Then
${}^L\bold G=(GL_2(\C)\times GL_2(\C))\rtimes {\rm Gal}(K/\Q)$, and $\bold G(\Bbb A)=GL_2(\Bbb A_K)$.
Let $\pi=\pi_{\infty}\otimes \otimes_p' \pi_p$ be a cuspidal representation of $\bold G(\Bbb A)$. 
Here $\pi_{\infty}$ is a unitary representation of $GL_2(\C)$. If $p$ splits in $K$ into $(v_1,v_2)$, then $\pi_p=\pi_{v_1}\otimes\pi_{v_2}$. We make the following assumption on $\pi$.

\begin{assump}\label{assump1} $\omega_{\pi}$ factors through $N_{K/\Q}$, i.e,
$\omega_{\pi}=\omega\circ N_{K/\Bbb Q}$ with a gr\"ossencharacter $\omega$.
\end{assump}

The automorphic induction corresponds to the $L$-group homomorphism
$$I_K^\Q: {}^L\bold G\longrightarrow GL(\C^2\oplus \C^2)\simeq GL_4(\Bbb C),\quad I_K^\Q(g,g';1)(x\oplus y)=g(x)\oplus g'(y), \quad 
I_K^\Q(1,1;\theta)(x\oplus y)=y\oplus x.
$$

Let $I_K^\Bbb Q \pi$ be the automorphic induction. It is automorphic representation of $GL_4/\Bbb Q$, and it is not cuspidal if and only if $\pi\simeq \pi\circ \theta$. In that case, $\pi$ is a base change of a cuspidal representation $\pi_0$ of $GL_2(\A_\Q)$, and
$I_K^\Bbb Q \pi=\pi_0\boxplus (\pi_0\otimes\omega_{K/\Bbb Q})$.

The Asai lift corresponds to the $L$-group homomorphism
$$As: {}^L\bold G\longrightarrow GL(\C^2\otimes \C^2)\simeq GL_4(\Bbb C),\quad As(g,g';1)(x\otimes y)=g(x)\otimes g'(y), \quad 
As(1,1;\theta)(x\otimes y)=y\otimes x.
$$

If $\rho: G_K\longrightarrow GL_2(\Bbb C)$, we have \cite{kim}
$$
\wedge^2 (\text{Ind}_K^\Q (\rho))= (As(\rho)\otimes \omega_{K/\Q}) \oplus \text{Ind}_K^\Q (\det\rho).
$$
Hence if $\rho$ corresponds to $\pi$, $det(\rho)$ corresponds to $\omega_{\pi}$. If $\pi$ satisfies Assumption \ref{assump1}, $I_K^\Q\, \omega_{\pi}=\omega\oplus \omega\omega_{K/\Bbb Q}$. Hence we can $L(s, \wedge^2(I_K^\Bbb Q)\otimes \chi^{-1})$ with $\chi=\omega$ or $\omega\omega_{K/\Bbb Q}$, has a pole at $s=1$, and 
$I_K^\Bbb Q \pi$ descends to a cuspidal representation of $GSp_4(\A_\Q)$ with the central character $\chi$ (cf. \cite{asgari&shahidi-spin}).

Let $\pi_{\infty}=\pi(1,1)$. Then the Langlands' parameter of $\pi_{\infty}$ is 
$$\phi: W_{\C}=\C^{\times}\longrightarrow (GL_2(\C)\times GL_2(\C))\rtimes {\rm Gal}(K/\Q),\quad \phi(z)=(I,I; \theta).
$$
So the Langlands' parameter of $I_K^\Q (\pi_{\infty})$ is
$$\phi: W_{\R}\longrightarrow GL_4(\C),\quad \phi(z)=Id,\quad \phi(j)=\begin{pmatrix} 0&I_2\\I_2&0\end{pmatrix}.
$$

Here if $P=\frac 12\begin{pmatrix} 1&0&1&0\\0&-1&0&1\\1&0&-1&0\\0&1&0&1\end{pmatrix}$, 
$$
P^{-1}\begin{pmatrix} 0&I_2\\I_2&0\end{pmatrix} P=\diag(1,-1,-1,1),\quad {}^t P \begin{pmatrix} 0&I_2\\-I_2&0\end{pmatrix} P=-\frac 12\begin{pmatrix} 0&I_2\\-I_2&0\end{pmatrix}.
$$ 
So $\begin{pmatrix} 0&I_2\\I_2&0\end{pmatrix}$ is conjugate to $\diag(1,-1,-1,1)$ in $GSp_4(\C)$, and up to conjugacy, we have
$$\phi: W_\R \lra GSp_4(\C),\quad \phi(z)=Id,\quad \phi(j)=\diag(1,-1,-1,1).
$$ 
Then $\phi$ is the Langlands' parameter for $\text{Ind}_B^{GSp_4}\, 1\otimes sgn\otimes sgn$, and as in \cite{kim-yam}, we can show that there exists a Siegel cusp form $F$ of weight (2,1) corresponding to $I_K^\Bbb Q \pi$. We have proved 

\begin{thm} Let $\pi$ be a cuspidal representation of $GL_2(\A_K)$, $K=\Bbb Q[\sqrt{-D}]$ which satisfies Assumption \ref{assump1},
and $\pi_\infty=\pi(1,1)$. Then there exists a non-holomorphic Siegel cusp form $F$ of weight $(2,1)$ such that $L(s,\pi_F)=L(s,\pi)$.
\end{thm}


\section{Non-holomorphic Siegel cusp forms of weight $(2,1)$ attached to Maass forms over $\Q$} 

Let $\pi$ be a cuspidal representation of $GL_2(\A_\Q)$ such that $\pi_\infty=\pi(1,1)$, i.e., Maass cusp form.
The Langlands parameter of $\pi_\infty$ is
$$\phi: W_{\R}\longrightarrow GL_2(\C),\quad \phi(z)=I_2,\quad \phi(j)=I_2.
$$
Let $BC(\pi)$ be the base change to $K=\Bbb Q[\sqrt{-D}]$, and consider 
$$\Pi=I_K^\Q (BC(\pi))=\pi\boxplus (\pi\otimes \omega_{K/\Q}).
$$
Then $\Pi$ descends to a generic cuspidal representation $\tau$ of $GSp_4(\A_\Q)$ (cf. \cite{asgari&shahidi-spin}).
The Langlands parameter of $\Pi_\infty$ is 

$$\phi: W_{\R}\longrightarrow GL_4(\C),\quad \phi(z)=Id, \quad \phi(j)=\diag(1, 1, -1, -1).
$$
Then we can show easily that for $s_2=\begin{pmatrix} 1&0&0&0\\0&0&0&1\\0&0&1&0\\0&-1&0&0\end{pmatrix}$, which is the Weyl group element corresponding to the long simple root, 
$$
s_2^{-1} \diag(1,1,-1,-1) s_2=\diag(1,-1,-1,1).
$$
Hence $\diag(1,1,-1,-1)$ and $\diag(1,-1,-1,1)$ are conjugate in $GSp_4(\C)$, and up to conjugacy, we have
$$\phi: W_\R \lra GSp_4(\C),\quad \phi(z)=Id,\quad \phi(j)=\diag(1,-1,-1,1).
$$ 
Since $\phi$ is the Langlands' parameter for $\text{Ind}_B^{GSp_4}\, 1\otimes sgn\otimes sgn$, as in \cite{kim-yam}, we can show that 
there exists a Siegel cusp form $F$ of weight (2,1) corresponding to $I_K^\Bbb Q (BC(\pi))$. We have proved 

\begin{thm} Let $\pi$ be a cuspidal representation of $GL_2(\A_\Q)$ such that $\pi_\infty=\pi(1,1)$. Then there exists a non-holomorphic Siegel cusp form $F$ of weight $(2,1)$ such that $L(s,\pi_F)=L(s,\pi)L(s,\pi\otimes\omega_{K/\Q})$.
\end{thm}


\section{Appendix}

Let $\pi=\otimes'_p \pi_p$ be a cuspidal representation of $GL_n(\A_\Q)$ with conductor $N$.
We show that the Langlands functoriality of the exterior $m$-th power $\wedge^m(\pi)$ as an automorphic representation of $GL_{C_{n,m}}(\Bbb A_\Q)$ implies Assumption (5). 

First, recall the following fact on the local $L$-factors of the ramified places:
Let $\Pi=\otimes'_p \Pi_p$ be a cuspidal representation of $GL_M(\A_\Q)$. Then for each prime $p$, $\Pi_p$ is unitary and generic. Recall that a non-tempered, unitary and generic representation of $GL_N(\Q_p)$ can be written as a full induced representation
\begin{equation}\label{rama}
{\rm Ind}\, \eta_1|\det|^{r_1}\otimes\cdots\otimes \eta_k|\det|^{r_k}\otimes\eta_0\otimes \eta_k|\det|^{-r_k}\otimes\cdots\otimes \eta_1|\det|^{-r_1},
\end{equation}
where $\eta_1,...,\eta_k$ are unitary square-integrable representations of $GL_{n_1}(\Q_p),...,GL_{n_k}(\Q_p)$, resp. and $\eta_0$ is a tempered representation of $GL_{n_0}(\Q_p)$ such that $n_0+n_1+\cdots+n_k=M$, and $0<r_k\leq \cdots \leq r_1\leq\frac 12-\frac 1{M^2+1}$ 
(\cite{T1}). (See \cite{LRS} for the bound.) Then by \cite{RS},
\begin{eqnarray}\label{L-fun}
&& L(s,\Pi_p)=L(s,\eta_0)\prod_{i=1}^k L(s\pm r_i,\eta_i), \nonumber\\
&& L(s,\Pi_p\times\widetilde\Pi_p)=\prod_{i,j=1}^k L(s\pm r_i\pm r_j, \eta_i\times \tilde\eta_j) \prod_{i=1}^k L(s\pm r_i, \eta_i\times\tilde\eta_0)L(s\pm r_i, \eta_0\times\tilde\eta_i).
\end{eqnarray}
Note that for any $i,j$, $L(s,\eta_i)$ and $L(s,\eta_i\times\tilde\eta_j)$ are holomorphic for $Re(s)>0$.

\medskip

For $1\leq m\leq n$, let $\wedge^m: GL_n(\Bbb C)\longrightarrow GL_{C_{n,m}}(\Bbb C)$ be the exterior $m$-th power, and let
$$
L(s,\pi,\wedge^m)=\sum_{n=1}^\infty \frac {\lambda_m(n)}{n^{s}}
$$
be the exterior $m$-th power $L$-function. Then it is easy to see that, for a prime $p\nmid N$,
$$\lambda_m(p)=a_m(p),\quad \text{for all $m=1,...,n$}.
$$

For each $p$, by the local Langlands' correspondence, $\wedge^m(\pi_p)$ is a well-defined representation of $GL_{C_{n,m}}(\Bbb Q_p)$: Let $\phi_p: W_{\Bbb Q_p}\times SL_2(\Bbb C)\longrightarrow GL_n(\Bbb C)$ be the parametrization of $\pi_p$. Then we have a map $\wedge^m(\phi_p): W_{\Bbb Q_p}\times SL_2(\Bbb C)\longrightarrow GL_{C_{n,m}}(\Bbb C)$. Then $\wedge^m(\pi_p)$ is the representation of $GL_{C_{n,m}}(\Bbb Q_p)$, corresponding to $\wedge^m(\phi_p)$.
Let $\wedge^m(\pi)=\otimes'_p \wedge^m(\pi_p)$. It is an irreducible admissible representation of $GL_{C_{n,m}}(\Bbb A_{\Bbb Q})$.

\begin{conjecture} (Langlands functoriality conjecture) The exterior $m$-th power $\Pi_m=\wedge^m(\pi)$ is an automorphic representation of $GL_{C_{n,m}}(\Bbb A_{\Bbb Q})$.
\end{conjecture} 

Since $\wedge^m(\pi)=\wedge^{n-m}\tilde\pi\otimes \omega$, it is enough to consider it for $m\leq [\frac n2]$. 
If $n\leq 3$, it is trivial. When $n=4$, it is proved in \cite{kim1}.

Suppose $\Pi_m$ is an automorphic representation of $GL_{C_{n,m}}(\Bbb A_{\Bbb Q})$. Consider the Rankin-Selberg $L$-function $L(s,\Pi_m\times\widetilde\Pi_m)$. 

\begin{prop} \label{RS}
There exists a holomorphic function $g(s)$ near $s=1$ such that
$$\log L(s,\Pi_m\times\widetilde\Pi_m)=\sum_{p\nmid N} \frac {|a_m(p)|^2}{p^s}+g(s),
$$ 
and if $s>1$,
$$\sum_{p\nmid N} \frac {|a_m(p)|^2}{p^s}\leq C_{n,m}^2\log\frac 1{s-1}+O(1),\quad \text{as $s\to 1^+$.}
$$
\end{prop}
\begin{proof} Let $r_m=C_{n,m}$, and $\displaystyle L(s,\Pi_m)=\prod_{p\nmid N} \prod_{i=1}^{r_m} (1-\beta_i(p)p^{-s})^{-1} \prod_{p|N} L_p(s)$, where $L_p(s)$ is a ramified factor as in (\ref{L-fun}). Then
$\displaystyle \log L(s,\Pi_m)=\sum_{l=1}^\infty \sum_{p\nmid N} \frac {b(p^l)}{l p^{ls}}+g'(s)$, where for $p\nmid N$, $b(p^l)=\beta_1(p)^l+\cdots+\beta_{r_m}(p)^l$. Hence
$b(p)=a_m(p)$ for $p\nmid N$. Here $g'(s)$ is a holomorphic function for $Re(s)>\frac 12$ by (\ref{L-fun}) and the fact that the local factors are non-vanishing.
Then we can easily see that

$$\log L(s,\Pi_m\times\widetilde\Pi_m)=\sum_{l=1}^\infty \sum_{p\nmid N} \frac {|b(p^l)|^2}{l p^{ls}}+g''(s)
=\sum_{p\nmid N} \frac {|a_m(p)|^2}{p^{s}}+\sum_{l=2}^\infty \sum_{p\nmid N} \frac {|b(p^l)|^2}{l p^{ls}}+g''(s),
$$
where $g''(s)$ is a holomorphic function for $Re(s)> 1-\frac 2{r_m^2+1}$ by (\ref{L-fun}).

Now we show that $\displaystyle \sum_{l=2}^\infty \sum_{p\nmid N} \frac {|b(p^l)|^2}{l p^{ls}}$ converges for $Re(s)> 1-\frac 2{r_{[\frac n2]}^2+1}$. By the classification of spherical unitary generic representation of $GL_n(\Bbb Q_p)$ \cite{T}, if $p\nmid N$, $\alpha_1(p),...,\alpha_n(p)$ is of the form
$$u_1 p^{a_1},\, u_2 p^{a_2},...,u_{k_p} p^{a_{k_p}},\, u_{k_p+1},...,u_{k_p+k_p'}, u_{k_p} p^{-a_{k_p}},...,u_1 p^{-a_1},
$$
where $u_i\in\Bbb C$ and $|u_i|=1$ for all $i$, and $0<a_{k_p}\leq \cdots \leq a_1\leq \frac 12-\frac 1{n^2+1}$. 
Now let $S_0$ be the set of primes where $\pi_p$ is tempered, and for $0<k\leq [\frac n2]$,
 let $S_k$ be the set of primes where $k_p=k$. Then
$\displaystyle \sum_{l=2}^\infty \sum_{p\in S_0} \frac {|b(p^l)|^2}{l p^{ls}}$ converges for $Re(s)>\frac 12$.

Now for each $0<k\leq [\frac n2]$, consider $\displaystyle \sum_{l=2}^\infty \sum_{p\in S_k} \frac {|b(p^l)|^2}{l p^{ls}}$.
Recall that $a_m(p)=\sum_{i_1<\cdots<i_m} \alpha_{i_1}(p)\cdots\alpha_{i_m}(p)$. Hence
$|b(p^l)|\ll p^{l(a_1+\cdots+a_m)}$, where we let $a_i=0$ if $i>k$. Also we have
$|a_k(p)|\gg p^{a_1+\cdots+a_k}$. Therefore, 
$$|b(p^l)|\ll |a_k(p)|^l.
$$
By assumption, $\wedge^k\pi$ is an automorphic representation of $GL_{r_k}/\Q$. Hence 
by \cite{LRS}, $|a_k(p)|\leq r_k p^{\frac 12-\frac 1{r_k^2+1}}$. Hence
$$\sum_{p\in S_k} \frac {|b(p^l)|^2}{l p^{ls}}\ll \sum_{p\in S_k} \frac {|a_k(p)|^2}{p^{l Re(s)-l+1+\frac {2(l-1)}{r_k^2+1}}}.
$$
Since $\wedge^k\pi$ is automorphic, by considering the Rankin-Selberg $L$-function $L(s,\wedge^k\pi\times \widetilde{\wedge^k\pi})$,
we have $\sum_p |a_k(p)|^2\ll x$. Hence by partial summation,
$$\sum_{p\in S_k} \frac {|a_k(p)|^2}{p^{l Re(s)-l+1+\frac {2(l-1)}{r_k^2+1}}}\ll 2^{-l Re(s)+l-\frac {2(l-1)}{r_k^2+1}}.
$$
Now $\displaystyle \sum_{l=2}^\infty 2^{-l Re(s)+l-\frac {2(l-1)}{r_k^2+1}}$ converges for $Re(s)>1-\frac 2{r_k^2+1}$. 
Hence our result follows.

Now
$L(s,\Pi_m\times\widetilde\Pi_m)$ has a pole at $s=1$, of order at least 1, and at most $r_m^2$. Also
$L(s,\Pi_m\times\widetilde\Pi_m)$ is zero free for $Re(s)>1-\frac c{(1+|t|)^n}$ for some constant $c>0$ (for example, \cite{B}).
Hence as $s\to 1^+$,
$$\log L(s,\Pi_m\times\widetilde\Pi_m)\leq r_m^2\log \frac 1{s-1}+O(1).
$$
This completes the proof.
\end{proof}

The above proof shows the following. (A different proof was given in \cite{cho-kim}.)

\begin{cor} Let $\displaystyle \log L(s,\pi)=\sum_{l=1}^\infty \sum_{p\nmid N} \frac {a(p^l)}{l p^{ls}}+g(s)$, where $a(p^l)=\alpha_1(p)^l+\cdots +\alpha_n(p)^l$ for $p\nmid N$, and $g(s)$ is a holomorphic function for $Re(s)>\frac 12$. Then the functoriality of $\wedge^m\pi$ for all $m$ implies Hypothesis $H$ in \cite{RS}, namely, for each $l\geq 2$,
$$\sum_{p\nmid N} \frac {|a(p^l)|^2 (\log p)^2}{p^{l}}
$$
converges.
\end{cor}

\end{document}